\documentclass[preprint,10pt]{elsarticle}
\usepackage{amsmath}
\usepackage{amsfonts}
\usepackage{amssymb}
\usepackage{graphicx}
\newtheorem{theorem}{Theorem}[section]

\newtheorem{definition}[theorem]{Definition}

\newtheorem{lemma}[theorem]{Lemma}

\newtheorem{proposition}[theorem]{Proposition}
\newtheorem{remark}[theorem]{Remark}

\newenvironment{proof}[1][Proof]{\noindent\textbf{#1.} }{\ \rule{0.5em}{0.5em}}
\numberwithin{equation}{section}

\begin{document}
\begin{frontmatter}
\title{Two laws of large numbers for sublinear expectations}
\author{Sun Chuanfeng\footnote{{This research is supported by NNSF project (11171187) of
China}.}}
\address{School of Mathematical Sciences, University of Jinan, Jinan, Shandong, 250022, P.R. China.}
\ead{sms\_suncf@ujn.edu.cn.}
\date{}

\date{}
\begin{abstract}
In this paper, we consider the sublinear expectation on bounded random variables. With the notion of uncorrelatedness for random variables under the sublinear expectation, a weak law of large numbers is obtained. With the notion of independence for random variable sequences and regular property for sublinear expectations, we get a strong one.
\end{abstract}

\end{frontmatter}

\section{Introduction}

In the last decades, the study of nonlinear laws
of large numbers (LLN) has been motivated
by its importance in decision theory,  mathematical finance and quantum mechanics. There are many papers related to LLNs for Choquet integrals. For examples, given a sequence $\{\xi_{i}\}_{i\geq 1}$ of IID random variables for a totally monotone capacity $\upsilon$, Maccheroni and Marinacci \cite{MM} obtain any cluster point of empirical averages lies between the lower Choquet integral $\int\xi_{1}d\upsilon$ and the upper Choquet integral $-\int-\xi_{1}d\upsilon$. That is
\begin{equation}\label{equation}
\upsilon(\{\omega\in\Omega; \int\xi_{1}d\upsilon\leq\mathop{\lim\inf}_{n}\frac{\sum_{i=1}^{n}\xi_{i}}{n}\leq\mathop{\lim\sup}_{n}\frac{\sum_{i=1}^{n}\xi_{i}}{n} \leq-\int-\xi_{1}d\upsilon\})=1.
\end{equation}
Epstein and Schneider \cite{ES} get the same result with the rectangular and 2-monotone properties for capacities. Besides Choquet integrals, the study of LLNs related to other nonlinear expectations (risk measures) has also been widely developed. For examples, the results similar as (1.1) are obtained by Cooman and Miranda \cite{CM} for coherent lower previsions and Chen, Wu and Li \cite{CWL} for sublinear expectations.  Though all these results are sharp and remarkable,  every cluster point lies in an interval is still far from the empirical average itself converges which is crucial for applications in economics and statistics.

Since all nonlinear expectations in above papers are either sublinear or superlinear, without loss of generality, we investigate LLNs for sublinear expectations\footnote{If $\Phi$ is a superlinear expectation , we can turn it into a sublinear one by defining $\mathcal{E}(\xi):=-\Phi(-\xi)$.}. Many improvements have been done and the most important one is we obtain there exists a constant sequence $\{\lambda_{i}\}_{i\geq 1}$ satisfying $\frac{1}{n^{p}}\sum_{i=1}^{n}(\xi_{i}-\lambda_{i})\rightarrow 0$, where we consider $p=1$ in the weak form and $p>1$ in the strong form.

For details, since every sublinear expectation $\mathcal{E}$ can be represented as the supremum of its dominated linear expectations $\{E_{\mu}; \mu\in\mathcal{M}\}$, we say $\xi_{1}$ and $\xi_{2}$ are uncorrelated with respect to $\mathcal{E}$ if they are uncorrelated under each dominated linear expectation $E_{\mu}$. For a pairwise uncorrelated sequence $\{\xi_{n}\}_{n\geq1}$, there exist constants $\{\lambda_{i}\}_{i\geq 1}$ such that for any $\epsilon>0$,
\[
\lim_{n\rightarrow\infty}C(|\frac{1}{n}\sum_{i=1}^{n}(\xi_{i}-\lambda_{i})|>\epsilon)=0,
\]
where $C$ is the nonlinear capacity derived from $\mathcal{E}$. Different from the classical result where $\lambda_{i}=E[\xi_{i}]$ is fixed,
here $\lambda_{i}$ is chosen from $[-\mathcal{E}(-\xi_{i}), \mathcal{E}(\xi_{i})]$ and the choice depends on the selected $\lambda_{1}, \cdots, \lambda_{i-1}$.

In order to get the strong form, the continuous assumption called regular property is appended to $\mathcal{E}$. Under the regular assumption for $\mathcal{E}$ and the independent assumption for $\{\xi_{i}\}_{i\geq 1}$, we prove for any $p>1$, there exist constants $\{\lambda_{i}\}_{i\geq 1}$ such that
\[
\frac{1}{n^{p}}\sum_{i=1}^{n}(\xi_{i}-\lambda_{i})\rightarrow 0,\quad C-q.s..
\]

The paper is divided into two parts. In section 2, we state basic notions and give out a weak LLN for pairwise uncorrelated sequences. In section 3, by showing the sum of an independent sequence converges $C$-q.s. is equivalent to converges in capacity, we induce the strong one.

\section{Basic notions and the weak LLN for sublinear expectations}

Through the paper, we denote $\phi$ as the empty set, $\mathbb{N}$ as natural numbers and $\mathbb{R}$ as real numbers. For a measurable space $(\Omega, \mathcal{F})$, we use $\mathcal{X}$ to denote the set of all bounded $\mathcal{F}$-measurable functions. If we endow $\mathcal{X}$ with the supremum norm, then it is a Banach space. The dual space of $\mathcal{X}$ is written as $\mathcal{X}^{\ast}$. Since there is a one-to-one correspondence between the element in $\mathcal{X}^{\ast}$ and the finitely additive set function, for the sake of convenience, we will not discriminate the element in $\mathcal{X}^{\ast}$ from its associated additive set function.
For a $\xi\in\mathcal{X}$, $\sigma(\xi)$ will denote the $\sigma$-field generated by $\xi$.
\begin{definition}\label{def-sublinear}
We say a functional $\mathcal{E}: \mathcal{X}\rightarrow\mathbb{R}$ is a sublinear expectation if it satisfies the following properties:

(i) Monotonicity: If $\xi_{1}\geq \xi_{2}$, $\mathcal{E}(\xi_{1})\geq\mathcal{E}(\xi_{2})$.

(ii) Constant preserving:  For any $\lambda\in \mathbb{R}$, $\mathcal{E}(\lambda)=\lambda$.

(iii) Subadditivity: For any $\xi_{1}, \xi_{2}\in\mathcal{X}$, $\mathcal{E}(\xi_{1}+\xi_{2})\leq \mathcal{E}(\xi_{1})+\mathcal{E}(\xi_{2})$.

(iv) Positive homogeneity: For any $\lambda\geq 0$ and $\xi\in\mathcal{X}$, $\mathcal{E}(\lambda \xi)=\lambda \mathcal{E}(\xi)$.
\end{definition}

\begin{theorem}
\label{theA.1} For a sublinear expectation $\mathcal{E}$, there exists a set $\mathcal{M}$ such that
\[
\mathcal{E}(\xi)=\max_{\mu\in\mathcal{M}}E_{\mu}[\xi],\quad \forall\xi\in\mathcal{X},
\]
where $\mu$ is a finitely additive set function.
\end{theorem}

\begin{proof}
We take $\{E_{\mu}; \mu\in\mathcal{M}\}$ as the set of linear expectations dominated by $\mathcal{E}$, i.e. $E_{\mu}[\xi]\leq\mathcal{E}(\xi)$ for any $\xi\in\mathcal{X}$. By the Corollary 2.4 of Chapter I in \cite{Simons}, for any $\xi\in\mathcal{X}$,
there exists a linear expectation $E_{\mu}$ such that $E_{\mu}[\xi]=\mathcal{E}(\xi)$. Then
\[
\mathcal{E}(\xi)=\max_{\mu\in\mathcal{M}}E_{\mu}[\xi],\quad \forall\xi\in\mathcal{X}.
\]
\end{proof}

Let us denote
\[
C(A):=\mathcal{E}(I_{A})=\max_{\mu\in\mathcal{M}}\mu(A)\quad \textrm{and}\quad c(A):=-\mathcal{E}(-I_{A})=\min_{\mu\in\mathcal{M}}\mu(A).
\]

\begin{definition}\label{def-q.s.}
The set $A$ is a $C$-polar set if $C(A)=0$ and we say $\xi_{n}\rightarrow \xi$, $C$-q.s. if $\xi_{n}\rightarrow\xi$ pointwisely outside a $C$-polar set.
\end{definition}

\begin{definition}\label{def-convergence-capacity}
We say $\xi_{n}\rightarrow \xi$ in capacity if for any $\epsilon, \delta>0$, there exists an $N\in\mathbb{N}$ such that $C(|\xi_{n}-\xi|>\epsilon)<\delta$ for any $n\geq N$.
\end{definition}
\begin{lemma}[Markov's inequality]\label{Markov}
If $\mathcal{E}$ is a sublinear expectation, for any $r>0$ and $\xi\in\mathcal{X}$ such that $\mathcal{E}(|\xi|^{p})<\infty$($p\geq 1$), we have
$C(|\xi|>r)\leq \frac{\mathcal{E}(|\xi|^{p})}{r^{p}}$.
\end{lemma}

\begin{proof} For any $E_{\mu}$ dominated by $\mathcal{E}$, we have
\[
E_{\mu}[I_{\{|\xi|>r\}}]\leq E_{\mu}[\frac{|\xi|^{p}}{r^{p}}I_{\{|\xi|>r\}}]\leq \frac{E_{\mu}[|\xi|^{p}]}{r^{p}}.
\]
Then
\[
E_{\mu}[I_{\{|\xi|>r\}}]\leq\frac{\mathcal{E}(|\xi|^{p})}{r^{p}}\quad \textrm{and}\quad C(\{|\xi|>r\})\leq\frac{\mathcal{E}(|\xi|^{p})}{r^{p}}.
\]
\end{proof}
\begin{definition}\label{uncorrelated}
For $\xi, \eta\in\mathcal{X}$ and sublinear expectation $\mathcal{E}$, we say $\xi$ and $\eta$ are uncorrelated with respect to $\mathcal{E}$ if $E_{\mu}[\xi\eta]=E_{\mu}[\xi]E_{\mu}[\eta]$ for any $\mu\in\mathcal{M}$, where $\mathcal{M}$ is defined as in Theorem \ref{theA.1}.
\end{definition}

\begin{remark}\label{remark}
For any constants $\alpha$, $\beta$ and $\lambda$, if $\xi$ is uncorrelated both with $\eta_{1}$ and $\eta_{2}$ with respect to $\mathcal{E}$, then $\lambda\xi+\alpha$ and $\eta_{1}+\eta_{2}+\beta$ are uncorrelated with respect to $\mathcal{E}$.
\end{remark}
\begin{lemma}\label{main-aux-general}
If $\mathcal{E}$ is a sublinear expectation, $\xi$ and $\eta$ are uncorrelated with respect to $\mathcal{E}$, then there exists a constant $\hat{\lambda}\in\mathcal{D}$ such that
\[
\mathcal{E}[(\xi-\hat{\lambda})\eta]=0,
\]
where $\mathcal{D}:=[-\mathcal{E}(-\xi), \mathcal{E}(\xi)]$.
\end{lemma}
\begin{proof}
It is easy to check $\mathcal{D}=\{E_{\mu}[\xi]; \mu\in\mathcal{M}\}$. Define $f(\lambda):=\mathcal{E}[(\xi-\lambda)\eta]$ and consider the optimization problem
\[
\inf_{\lambda\in\mathcal{D}}f(\lambda).
\]

On one hand, for any $\lambda'\in\mathcal{D}$, there exists a $\mu'\in\mathcal{M}$ such that $\lambda'=E_{\mu'}[\xi]$. We have
\[
\mathcal{E}[(\xi-\lambda')\eta]\geq E_{\mu'}[(\xi-\lambda')\eta]=E_{\mu'}[\xi-\lambda']E_{\mu'}[\eta]=0.
\]
Then
\[
\inf_{\lambda\in\mathcal{D}}f(\lambda)=\inf_{\lambda\in\mathcal{D}}\mathcal{E}[(\xi-\lambda)\eta]\geq0.
\]
On the other hand, by Theorem \ref{MO}, there exists a $\hat{\mu}\in\mathcal{M}$ such that
\[
\inf_{\lambda\in\mathcal{D}}E_{\hat{\mu}}[(\xi-\lambda)\eta]=\inf_{\lambda\in\mathcal{D}}\mathcal{E}[(\xi-\lambda)\eta].
\]
Since
\[
\inf_{\lambda\in\mathcal{D}}E_{\hat{\mu}}[(\xi-\lambda)\eta]\leq E_{\hat{\mu}}[(\xi-E_{\hat{\mu}}[\xi])\eta]=E_{\hat{\mu}}[\xi-E_{\hat{\mu}}[\xi]]E_{\hat{\mu}}[\eta]=0.
\]
Then
\[
\inf_{\lambda\in\mathcal{D}}f(\lambda)=0.
\]

For any $\lambda_{1}, \lambda_{2}\in\mathbb{R}$, we have
\[
|f(\lambda_{1})-f(\lambda_{2})|=|\mathcal{E}[(\xi-\lambda_{1})\eta]-\mathcal{E}[(\xi-\lambda_{2})\eta]|\leq|\mathcal{E}(\eta)||\lambda_{1}-\lambda_{2}|.
\]
$f$ is a continuous function on $\mathbb{R}$ due to $\eta$ is a bounded random variable. With $\mathcal{D}$ is a closed set in $\mathbb{R}$, the infimum of $f$ can be obtained by some $\hat{\lambda}\in\mathcal{D}$.
\end{proof}

\begin{theorem}[The weak LLN]\label{general-WEAK LLN}
Suppose $\mathcal{E}$ is a sublinear expectation and $\{\xi_{n}\}_{n\geq 1}\subset\mathcal{X}$ is a pairwise uncorrelated sequence with respect to $\mathcal{E}$. If $\sup_{n\geq 1}(\mathcal{E}[\xi_{n}-\mathcal{E}(\xi_{n})]^{2}+[\mathcal{E}(\xi_{n})+\mathcal{E}(-\xi_{n})]^{2})<\infty$, there exist constants $\{\lambda_{i}\}_{i\geq 1}$ such that $\lambda_{i}\in[-\mathcal{E}(-\xi_{i}), \mathcal{E}(\xi_{i})]$ and for any $\epsilon>0$,
\[
\lim_{n\rightarrow\infty}C(|\frac{1}{n}\sum_{i=1}^{n}(\xi_{i}-\lambda_{i})|>\epsilon)=0.
\]
\end{theorem}
\begin{proof}
Take $\lambda_{1}=\mathcal{E}(\xi_{1})$. Since $\{\xi_{n}\}_{n\geq 1}\subset\mathcal{X}$ is pairwise uncorrelated with respect to $\mathcal{E}$, by Remark \ref{remark} and Lemma \ref{main-aux-general}, there exist constants $\{\lambda_{i}; i\geq 2\}$ such that $\lambda_{i}\in[-\mathcal{E}(-\xi_{i}), \mathcal{E}(\xi_{i})]$ and
\[
\mathcal{E}[(\xi_{i}-\lambda_{i})(\sum_{j=1}^{i-1}(\xi_{j}-\lambda_{j}))]=0, \quad\forall i\geq 2.
\]
By the Markov's inequality, we have
\[
\begin{array}{r@{}l}
& C(|\frac{1}{n}\sum_{i=1}^{n}(\xi_{i}-\lambda_{i})|>\epsilon)\leq\mathcal{E}(\frac{[\sum_{i=1}^{n}(\xi_{i}-\lambda_{i})]^{2}}{n^{2}\epsilon^{2}})\\
\leq  & \frac{1}{n^{2}\epsilon^{2}}[\mathcal{E}(\xi_{n}-\lambda_{n})^{2}+2\mathcal{E}[(\xi_{n}-\lambda_{n})(\sum_{i=1}^{n-1}(\xi_{i}-\lambda_{i}))]+\mathcal{E}(\sum_{i=1}^{n-1}(\xi_{i}-\lambda_{i}))^{2}]\\
\leq & \frac{1}{n^{2}\epsilon^{2}}\sum_{i=1}^{n}\mathcal{E}(\xi_{i}-\lambda_{i})^{2}.
\end{array}
\]
Due to $\lambda_{i}\in[-\mathcal{E}(-\xi_{i}), \mathcal{E}(\xi_{i})]$, then
\[
\mathcal{E}(\xi_{i}-\lambda_{i})^{2}\leq\mathcal{E}[\xi_{n}-\mathcal{E}(\xi_{n})]^{2}+[\mathcal{E}(\xi_{n})+\mathcal{E}(-\xi_{n})]^{2}.
\]
We have
\[
C(|\frac{1}{n}\sum_{i=1}^{n}(\xi_{i}-\lambda_{i})|>\epsilon)\leq \frac{1}{n\epsilon^{2}}\sup_{n\geq 1}(\mathcal{E}[\xi_{n}-\mathcal{E}(\xi_{n})]^{2}+[\mathcal{E}(\xi_{n})+\mathcal{E}(-\xi_{n})]^{2}).
\]
As $n\rightarrow\infty$, $C(|\frac{1}{n}\sum_{i=1}^{n}(\xi_{i}-\lambda_{i})|>\epsilon)\rightarrow 0$.
\end{proof}

\section{The strong LLN for sublinear expectations with regular property}

In this section, in order to get the strong LLN, on one hand, the regular property is appended to sublinear expectations. On the other hand, for the random variable sequence, uncorrelatedness is strengthened as independence.
\subsection{Regularity and independence}
\begin{definition}
We say a sublinear expectation $\mathcal{E}$ is regular if for any sequence $\{\xi_{n}\}_{n\geq1}\subset\mathcal{X}$ such that $\xi_{n}\downarrow0$, we have $\mathcal{E}(\xi_{n})\downarrow0$.
\end{definition}

\begin{lemma}
\label{lem-fatou-probility-property} If the sublinear expectation $\mathcal{E}$ is
regular, for any linear expectation $E_{\mu}$
dominated by $\mathcal{E}$, $\mu$ is a probability measure.
\end{lemma}

\begin{proof}
For any $A_{n}\downarrow\phi$, we have $\mathcal{E}(I_{A_{n}})\downarrow0$. If a
linear expectation $E_{\mu}$ is dominated by $\mathcal{E}$, then $\mu(A_{n})\downarrow0$. It
is easy to see that $\mu(\Omega)=1$.
\end{proof}

In the following, if $\mathcal{E}$ is regular, we will use $\mathcal{P}$ to replace $\mathcal{M}$ to denote all the linear expectations
dominated by $\mathcal{E}$.

\begin{proposition}\label{capacity-regular}
If the sublinear expectation $\mathcal{E}$ is regular, for any set sequence $\{A_{n}\}_{n\geq1}$ such that $A_{n}\uparrow A$ ($A_{n}\downarrow A$), we have
\[
C(A)=\lim_{n\rightarrow\infty}C(A_{n}).
\]
\end{proposition}
\begin{proof}
If $A_{n}\uparrow A$, by Lemma \ref{lem-fatou-probility-property}, we have
\[
C(A)=\sup_{P\in\mathcal{P}}P(A)=\sup_{P\in\mathcal{P}}\sup_{n\in\mathbb{N}}P(A_{n})=\sup_{n\in\mathbb{N}}\sup_{P\in\mathcal{P}}P(A_{n})=\sup_{n\in\mathbb{N}}C(A_{n})=\lim_{n\rightarrow\infty}C(A_{n}).
\]

If $A_{n}\downarrow A$,
\[
0\leq C(A_{n})-C(A)\leq\mathcal{E}(I_{A_{n}}-I_{A}).
\]
Since $I_{A_{n}}-I_{A}\downarrow0$, we have $\lim_{n\rightarrow\infty}\mathcal{E}(I_{A_{n}}-I_{A})=0$. Then $\lim_{n\rightarrow\infty}C(A_{n})=C(A)$.
\end{proof}

\begin{definition}[Independent sequence]\label{weakly independence}
For a sequence $\{\xi_{n}\}_{n\geq 1}\subset\mathcal{X}$, we say it is independent under the sublinear expectation $\mathcal{E}$ if it is pairwise uncorrelated and for any $n>m$, $A\in\sigma(\xi_{1}, \xi_{2}, \cdots, \xi_{m})$ and $B\in\sigma(\xi_{m+1},\xi_{m},\cdots, \xi_{n})$, we have
\[
C(AB)\leq C(A)C(B).
\]
\end{definition}

\subsection{The strong LLN for independent sequences}

For a sequence $\{\xi_{n}\}_{n\geq 1}$, denote $S_{n}=\sum_{i=1}^{n}\xi_{i}$. Some results we used in this subsection are similar as in the classical case for probability measures, we put them in Appendix B.
\begin{lemma}[Extended Ottaviani's inequality]\label{Ottaviani}
Suppose $r, s, t$ are positive numbers and $\{\xi_{n}\}_{n\geq 1}$ is an independent sequence. If for $0\leq k\leq n-1$, we have
\[
c(|\sum_{i=k+1}^{n}\xi_{i}|\leq s)\geq r.
\]
Then
\[
\begin{array}{r@{}l}
C(\max_{1\leq k\leq n}|\sum_{i=1}^{k}\xi_{i}|>s+t)\leq C(|\sum_{i=1}^{n}\xi_{i}|>t)+(1-r)\sum_{k=1}^{n}C(A_{k}),
\end{array}
\]
where $A_{k}=\{\omega; \max_{1\leq i\leq k-1}|S_{i}|\leq s+t, |S_{k}|>s+t\}$.
\end{lemma}
\begin{proof}
Denote
\[
B_{k}=\{\omega; |S_{n}-S_{k}|\leq s\}=\{\omega; |\sum_{i=k+1}^{n}\xi_{i}|\leq s\}.
\]
Since $\{A_{k}\}_{1\leq k\leq n}$ is disjoint with each other and $\sum_{k=1}^{n}A_{k}B_{k}\subset\{|S_{n}|>t\}$, we have
\[
\begin{array}{r@{}l}
C(|\sum_{i=1}^{n}\xi_{i}|>t) & \geq C(\sum_{k=1}^{n}A_{k}B_{k})=\mathcal{E}[\sum_{k=1}^{n}I_{A_{k}}(1-I_{B^{c}_{k}})]\\
 & \geq C(\sum_{k=1}^{n}A_{k})-C(\sum_{k=1}^{n}A_{k}B^{c}_{k})\\
 & \geq C(\sum_{k=1}^{n}A_{k})-\sum_{k=1}^{n}C(A_{k}B^{c}_{k})\\
  & \geq C(\sum_{k=1}^{n}A_{k})-\sum_{k=1}^{n}C(A_{k})C(B^{c}_{k})\\
 & = C(\sum_{k=1}^{n}A_{k})-\sum_{k=1}^{n}C(A_{k})(1-c(B_{k}))\\
\end{array}
\]
Since $c(B_{k})\geq r$, then
\[
\begin{array}{r@{}l}
 & C(\max_{1\leq k\leq n}|\sum_{i=1}^{k}\xi_{i}|>s+t)\\
= & C(\sum_{k=1}^{n}A_{k})\leq C(|\sum_{i=1}^{n}\xi_{i}|>t)+(1-r)\sum_{k=1}^{n}C(A_{k}).
\end{array}
\]
\end{proof}

In order to get the strong LLN, we need the following assumption:

($H_{0}$) For any set sequence $\{A_{n}\}_{n\geq1}$ such that $\bigcup_{n\geq1}A_{n}=\Omega$ and $A_{i}\bigcap A_{j}=\phi$ for $i\not=j$, there exists a constant $M$ such that $\sum_{n=1}^{\infty}C(A_{n})<M$.

\begin{proposition}\label{pro-capacity-q.s.}
Suppose $\mathcal{E}$ is a regular sublinear expectation satisfying Assumption ($H_{0}$) and $\{\xi_{n}\}_{n\geq1}$ is an independent sequence. Then $\sum_{i=1}^{n}\xi_{i}$ converges to some finite valued random variable $C$-q.s. if and only if it converges in capacity.
\end{proposition}
\begin{proof}
If $\sum_{i=1}^{n}\xi_{i}$ converges $C$-q.s., it converges in capacity by Theorem \ref{converge q.s.} and Theorem \ref{Chauchy-convergence}.

On the other hand, if $\sum_{i=1}^{n}\xi_{i}$ converges in capacity, by Theorem \ref{Chauchy-convergence}, for any $\epsilon>0$ and $\delta>0$, there exists an $N\in\mathbb{N}$ such that for any $m, n>N$,
\[
C(|\sum_{i=m}^{n}\xi_{i}|>\epsilon)<\delta \quad \textrm{and} \quad c(|\sum_{i=m}^{n}\xi_{i}|\leq\epsilon)\geq1-\delta.
\]

Take $r=1-\delta$ and $s=t=\epsilon$ in Lemma \ref{Ottaviani}. Then
\[
C(\sup_{1\leq k\leq n}|\sum_{i=N+1}^{N+k}\xi_{i}|>2\epsilon)\leq C(|\sum_{i=N+1}^{N+n}\xi_{i}|>\epsilon)+\delta\sum_{k=1}^{n}C(A_{k})<(1+M)\delta,
\]
where $A_{k}=\{\omega; \max_{1\leq i\leq k-1}|S_{N+i}-S_{N}|\leq 2\epsilon, |S_{N+k}-S_{N}|>2\epsilon\}$. Thus
\[
C(\sup_{k\geq 1}|\sum_{i=N+1}^{N+k}\xi_{i}|>2\epsilon)<(1+M)\delta
\]
and
\[
C(\sup_{m, n>N}|\sum_{i=m}^{n}\xi_{i}|>4\epsilon)\leq C(\sup_{m>N}|\sum_{i=N+1}^{m}\xi_{i}|>2\epsilon)+C(\sup_{n>N}|\sum_{i=N+1}^{n}\xi_{i}|>2\epsilon)\leq2(1+M)\delta.
\]
By Theorem \ref{converge q.s.}, $\sum_{i=1}^{n}\xi_{i}$ converges $C$-q.s..
\end{proof}
\begin{theorem}[The strong LLN]\label{STRONG LLN}
Suppose $\mathcal{E}$ is a regular sublinear expectation satisfying Assumption ($H_{0}$) and $\{\xi_{n}\}_{n\geq1}$ is an independent sequence. If $\sup_{i\geq 1}[\mathcal{E}(\xi_{i})+\mathcal{E}(-\xi_{i})]<\infty$ and $\sum_{i=1}^{\infty}\frac{1}{i^{2p}}\mathcal{E}(\xi_{i}-\mathcal{E}(\xi_{i}))^{2}<\infty$, then for any $p>1$, there exist constants $\{\lambda_{i}\}_{i\geq 1}$ such that
\[
\lim_{n\rightarrow\infty}\frac{1}{n^{p}}\sum_{i=1}^{n}(\xi_{i}-\lambda_{i})=0,\quad C-q.s..
\]
\end{theorem}
\begin{proof}
Take $\lambda_{1}:=\mathcal{E}(\xi_{1})$. $\{\xi_{n}\}_{n\geq1}$ is an independent sequence induces $\{\frac{\xi_{i}}{i^{p}}\}_{i\geq 1}$ is pairwise uncorrelated. By Remark \ref{remark} and Lemma \ref{main-aux-general}, there exists constants $\{\lambda_{i}\}_{i\geq 1}$ such that $\lambda_{i}\in[-\mathcal{E}(-\xi_{i}), \mathcal{E}(\xi_{i})]$ and
\[
\mathcal{E}[\frac{1}{j^{p}}(\xi_{j}-\lambda_{j})(\sum_{i=1}^{j-1}\frac{1}{i^{p}}[\xi_{i}-\lambda_{i}])]=0, \forall j\geq 2.
\]
Then for any $n>m$, we have
\[
\begin{array}{r@{}l}
& C(|\sum_{i=m}^{n}\frac{1}{i^{p}}(\xi_{i}-\lambda_{i})|>\epsilon)\\
\leq & \mathcal{E}[\frac{(\sum_{i=m}^{n}\frac{1}{i^{p}}[\xi_{i}-\lambda_{i}])^{2}}{\epsilon^{2}}]\\
= & \frac{1}{\epsilon^{2}}\mathcal{E}[\sum_{i=m}^{n}\frac{1}{i^{2p}}(\xi_{i}-\lambda_{i})^{2}+2\sum_{j=m+1}^{n}\frac{1}{j^{p}}(\xi_{j}-\lambda_{j})(\sum_{i=m}^{j-1}\frac{1}{i^{p}}[\xi_{i}-\lambda_{i}])]\\
\leq & \frac{1}{\epsilon^{2}}[\sum_{i=m}^{n}\frac{1}{i^{2p}}\mathcal{E}(\xi_{i}-\lambda_{i})^{2}]+\frac{2}{\epsilon^{2}}\mathcal{E}[\sum_{j=m+1}^{n}\frac{1}{j^{p}}(\xi_{j}-\lambda_{j})(\sum_{i=m}^{j-1}\frac{1}{i^{p}}[\xi_{i}-\lambda_{i}])]\\
= & \frac{1}{\epsilon^{2}}[\sum_{i=m}^{n}\frac{1}{i^{2p}}\mathcal{E}(\xi_{i}-\lambda_{i})^{2}]\\
& +\frac{2}{\epsilon^{2}}\mathcal{E}[\sum_{j=m+1}^{n}\frac{1}{j^{p}}(\xi_{j}-\lambda_{j})(\sum_{i=1}^{j-1}\frac{1}{i^{p}}[\xi_{i}-\lambda_{i}]-\sum_{i=1}^{m-1}\frac{1}{i^{p}}[\xi_{i}-\lambda_{i}])]\\
\leq & \frac{1}{\epsilon^{2}}[\sum_{i=m}^{n}\frac{1}{i^{2p}}\mathcal{E}(\xi_{i}-\lambda_{i})^{2}]+\frac{2}{\epsilon^{2}}\mathcal{E}[\sum_{j=m+1}^{n}\frac{1}{j^{p}}(\xi_{j}-\lambda_{j})(-\sum_{i=1}^{m-1}\frac{1}{i^{p}}[\xi_{i}-\lambda_{i}])]\\
= & \frac{1}{\epsilon^{2}}[\sum_{i=m}^{n}\frac{1}{i^{2p}}\mathcal{E}(\xi_{i}-\lambda_{i})^{2}]+\frac{2}{\epsilon^{2}}\mathcal{E}[(\sum_{j=m+1}^{n}\frac{1}{j^{p}}[\xi_{j}-\lambda_{j}])(-\sum_{i=1}^{m-1}\frac{1}{i^{p}}[\xi_{i}-\lambda_{i}])].
\end{array}
\]
By Remark \ref{remark} and Lemma \ref{main-aux-general}, there exists a constant $\lambda_{n, m}$ such that
\[
\mathcal{E}[(\sum_{j=m+1}^{n}\frac{\xi_{j}}{j^{p}}-\lambda_{n, m})(-\sum_{i=1}^{m-1}\frac{1}{i^{p}}[\xi_{i}-\lambda_{i}])]=0
\]
and $\lambda_{n, m}\in[-\mathcal{E}(-(\sum_{j=m+1}^{n}\frac{\xi_{j}}{j^{p}})), \mathcal{E}(\sum_{j=m+1}^{n}\frac{\xi_{j}}{j^{p}})]$, which induces
\[
\sum_{j=m+1}^{n}-\mathcal{E}(-(\frac{\xi_{j}}{j^{p}}))\leq \lambda_{n, m}\leq \sum_{j=m+1}^{n}\mathcal{E}(\frac{\xi_{j}}{j^{p}}).
\]
Since $\sup_{i\geq 1}[\mathcal{E}(\xi_{i})+\mathcal{E}(-\xi_{i})]<\infty$, there exists $M>0$ such that $\sup_{i\geq 1}[\mathcal{E}(\xi_{i})+\mathcal{E}(-\xi_{i})]\leq M$. Then
\[
\begin{array}{r@{}l}
 & \mathcal{E}[(\sum_{j=m+1}^{n}\frac{1}{j^{p}}[\xi_{j}-\lambda_{j}])(-\sum_{i=1}^{m-1}\frac{1}{i^{p}}[\xi_{i}-\lambda_{i}])]\\
= & \mathcal{E}[(\sum_{j=m+1}^{n}\frac{1}{j^{p}}[\xi_{j}-\lambda_{j}])(-\sum_{i=1}^{m-1}\frac{1}{i^{p}}[\xi_{i}-\lambda_{i}])]\\
 & -\mathcal{E}[(\sum_{j=m+1}^{n}\frac{\xi_{j}}{j^{p}}-\lambda_{n,m})(-\sum_{i=1}^{m-1}\frac{1}{i^{p}}[\xi_{i}-\lambda_{i}])]\\
\leq & \mathcal{E}[\lambda_{n,m}-\sum_{j=m+1}^{n}\frac{\lambda_{j}}{j^{p}})(-\sum_{i=1}^{m-1}\frac{1}{i^{p}}[\xi_{i}-\lambda_{i}])]\\
\leq & (\sum_{j=m+1}^{n}\frac{M}{j^{p}})(\sum_{i=1}^{m-1}\frac{M}{i^{p}})
\end{array}
\]

Since the sequence $\{\sum_{i=1}^{n}\frac{1}{i^{p}}\}_{n\geq 1}$ is convergent when $p>1$, we have
\[
\lim_{n, m\rightarrow\infty}(\sum_{j=m+1}^{n}\frac{M}{j^{p}})(\sum_{i=1}^{m-1}\frac{M}{i^{p}})=0.
\]

Since
\[
\begin{array}{r@{}l}
 & \sum_{i=1}^{\infty}\frac{1}{i^{2p}}\mathcal{E}(\xi_{i}-\lambda_{i})^{2}\\
\leq & \sum_{i=1}^{\infty}\frac{1}{i^{2p}}\mathcal{E}(\xi_{i}-\mathcal{E}(\xi_{i}))^{2}+\sum_{i=1}^{\infty}\frac{1}{i^{2p}}(\mathcal{E}(\xi_{i})+\mathcal{E}(-\xi_{i}))^{2}<\infty,
\end{array}
\]
we have
\[\lim_{n, m\rightarrow\infty}\sum_{i=m}^{n}\frac{1}{i^{2p}}\mathcal{E}(\xi_{i}-\lambda_{i})^{2}=0.
\]
It shows $\{S_{n}'; S_{n}'=\sum_{i=1}^{n}\frac{1}{i^{p}}(\xi_{i}-\lambda_{i})\}_{n\geq1}$  is a Cauchy sequence in capacity. By Theorem \ref{Chauchy-convergence}, it converges in capacity.

Since $\{\frac{1}{n^{p}}(\xi_{n}-\lambda_{n})\}_{n\geq1}$ is an independent sequence, by Proposition \ref{pro-capacity-q.s.}, $\{S_{n}'\}_{n\geq 1}$ converges $C$-q.s.. With the Kronecker's lemma, we have
\[
\lim_{n\rightarrow\infty}\frac{1}{n^{p}}\sum_{i=1}^{n}(\xi_{i}-\lambda_{i})=0 \quad C-q.s..
\]
\end{proof}
\appendix

\section{Some basic results}

In this part, we restate the Mazur-Orlicz theorem used in this paper.
\begin{theorem}
[Mazur-Orlicz theorem]\label{MO} Suppose $\mathcal{X}$ is a nonzero space,
$\mathcal{E}:\mathcal{X}\rightarrow\mathbb{R}$ is sublinear and $\mathcal{D}$ is a
nonempty convex subset of $\mathcal{X}$. Then there exists a linear functional
$L$ on $\mathcal{X}$ such that $L$ is dominated by $\mathcal{E}$ and
\[
\inf_{\xi\in\mathcal{D}}L(\xi)=\inf_{\xi\in\mathcal{D}}\mathcal{E}(\xi).
\]

\end{theorem}

\begin{proof}
Refer the Lemma 1.6 of Chapter 1 in \cite{Simons}.
\end{proof}

\section{The results used in proving the strong LLN}
\begin{theorem}\label{converge q.s.}
If the sublinear expectation $\mathcal{E}$ is regular, for a sequence $\{\xi_{i}\}_{i\geq 1}$, $S_{n}=\sum_{i=1}^{n}\xi_{i}$ converges to some finite valued random variable $S$, $C$-q.s. if and only if for any $\epsilon>0$, we have
\[
\lim_{n\rightarrow\infty} C(\sup_{i, j> n}|\sum_{m=i}^{j}\xi_{m}|>\epsilon)=0.
\]
\end{theorem}
\begin{proof}
Since $\{\sup_{i, j> n}|\sum_{m=i}^{j}\xi_{m}|>\epsilon\}=\bigcup_{i ,j> n}\{|\sum_{m=i}^{j}\xi_{m}|>\epsilon\}$, by Proposition \ref{capacity-regular}, we have
\[
\lim_{n\rightarrow\infty}C(\sup_{i, j> n}|\sum_{m=i}^{j}\xi_{m}|>\epsilon)=\lim_{n\rightarrow\infty} C(\bigcup_{i ,j> n}\{|\sum_{m=i}^{j}\xi_{m}|>\epsilon\})=C(\bigcap_{n=1}\bigcup_{i ,j> n}\{|\sum_{m=i}^{j}\xi_{m}|>\epsilon\}).
\]

$\Rightarrow$ For any $\epsilon>0$ and $\omega\in\{\lim_{n\rightarrow\infty}S_{n}(\omega)=S(\omega)\}$, there exists an $N(\epsilon, \omega)$ such that $|S_{j}-S_{i}|=|\sum_{m=i}^{j}\xi_{m}|\leq\epsilon$ for any $i, j>N(\epsilon, \omega)$, i.e.
\[
\omega\not\in\bigcap_{n=1}\bigcup_{i ,j> n}\{|\sum_{m=i}^{j}\xi_{m}|>\epsilon\}.
\]
Then
\[
\bigcap_{n=1}\bigcup_{i ,j> n}\{|\sum_{m=i}^{j}\xi_{m}|>\epsilon\}\subset\{\lim_{n\rightarrow\infty}S_{n}(\omega)\not=S(\omega)\}.
\]
Since $C(\{\lim_{n\rightarrow\infty}S_{n}(\omega)\not=S(\omega)\})=0$, then
\[
\begin{array}{r@{}l}
\lim_{n\rightarrow\infty}C(\sup_{i, j> n}|\sum_{m=i}^{j}\xi_{m}|>\epsilon) & =C(\bigcap_{n=1}\bigcup_{i ,j> n}\{|\sum_{m=i}^{j}\xi_{m}|>\epsilon\})\\
 & \leq C(\{\lim_{n\rightarrow\infty}S_{n}(\omega)\not=S(\omega)\})=0.
\end{array}
\]
$\Leftarrow$ Take $A_{k}:=\bigcap_{n=1}\bigcup_{i ,j\geq n}\{|\sum_{m=i}^{j}\xi_{m}|>\frac{1}{k}\}$. Then $C(A_{k})=0$ for any $k\geq 1$. By Proposition \ref{capacity-regular}, we have
\[
C(\bigcup_{k\geq 1}A_{k})=\lim_{k\rightarrow\infty}C(A_{k})=0.
\]
For any $\omega\not\in \bigcup_{k\geq 1}A_{k}$,
\[
\omega\in A^{c}_{k}=\bigcup_{n=1}\bigcap_{i ,j\geq n}\{|\sum_{m=i}^{j}\xi_{m}|\leq\frac{1}{k}\}, \quad \forall k\geq 1.
\]
For any $\epsilon>0$, choose a $k$ such that $\frac{1}{k}<\epsilon$. There exists an $N(\epsilon, \omega)$ such that $\omega\in\{|\sum_{m=i}^{j}\xi_{m}|\leq\frac{1}{k}\}$ for any $i, j>N(\epsilon, \omega)$, i.e.
\[
|\sum_{m=i}^{j}\xi_{m}(\omega)|\leq\frac{1}{k}<\epsilon, \quad\forall i, j>N(\epsilon, \omega).
\]
\end{proof}
\begin{theorem}\label{lem-subsequence}
If the sublinear expectation $\mathcal{E}$ is regular and $\{\xi_{n}\}_{n\geq 1}$ is a Cauchy sequence in capacity, there exists a subsequence $\{\xi_{n_{k}}\}_{k\geq1}$ converges to some $\xi$, $C$-q.s..
\end{theorem}

\begin{proof} Consider the subsequence $\{\xi_{n_{k}}\}_{k\geq1}$ such that
\[
C(|\xi_{n_{k+1}}-\xi_{n_{k}}|>\frac{1}{2^{k}})<\frac{1}{2^{k}}.
\]
Take $A_{k}:=\{|\xi_{n_{k+1}}-\xi_{n_{k}}|>\frac{1}{2^{k}}\}$.
Since $\mathcal{E}$ is regular, we have
\[
C(\bigcup^{\infty}_{i=k}A_{i})=\lim_{n\rightarrow\infty}C(\bigcup^{n}_{i=k}\{|\xi_{n_{i+1}}-\xi_{n_{i}}|>\frac{1}{2^{i}}\})\leq \sum_{i=k}^{\infty}C(|\xi_{n_{i+1}}-\xi_{n_{i}}|>\frac{1}{2^{i}})\leq\frac{1}{2^{k-1}}.
\]
Then
\[
C(\limsup A_{k})=\lim_{k\rightarrow\infty}C(\bigcup^{\infty}_{i=k}A_{i})=0.
\]

For each $\omega\in(\limsup A_{k})^{c}$, there exists an $N(\omega)$ such that $|\xi_{n_{k+1}}(\omega)-\xi_{n_{k}}(\omega)|\leq \frac{1}{2^{k}}$ for any $k\geq N(\omega)$. We have $\sum_{k=1}^{\infty}|\xi_{n_{k}+1}(\omega)-\xi_{n_{k}}(\omega)|<\infty$. Then $\{\xi_{n_{k}}(\omega)\}_{k\geq 1}$ converges.
\end{proof}
\begin{theorem}\label{Chauchy-convergence}
If the sublinear expectation $\mathcal{E}$ is regular, $\{\xi_{n}\}_{n\geq 1}$ converges to $\xi$ in capacity if and only if it is a Cauchy sequence in capacity.
\end{theorem}
\begin{proof}
$\Rightarrow$ The result comes from for any $\epsilon>0$,
\[
C(|\xi_{n}-\xi_{m}|>\epsilon)\leq C(|\xi_{n}-\xi|>\frac{\epsilon}{2})+C(|\xi_{m}-\xi|>\frac{\epsilon}{2}).
\]

$\Leftarrow$ By Theorem \ref{lem-subsequence}, choose a subsequence $\{\xi_{n_{k}}\}_{k\geq 1}$ such that it converges to some $\xi$, $C$-q.s.. By using the same method as in Theorem \ref{converge q.s.}, we can get $C(\bigcap_{i=1}\bigcup_{k\geq i}|\xi_{n_{k}}-\xi|>\epsilon)=0$ for any $\epsilon>0$. Then $\lim_{i\rightarrow\infty} C(\bigcup_{k\geq i}|\xi_{n_{k}}-\xi|>\epsilon)=0$ and $\{\xi_{n_{k}}\}_{k\geq 1}$ converges to $\xi$ in capacity. Then the result comes from for any $\epsilon>0$,
\[
C(|\xi_{n}-\xi|>\epsilon)\leq C(|\xi_{n}-\xi_{n_{k}}|>\frac{\epsilon}{2})+C(|\xi_{n_{k}}-\xi|>\frac{\epsilon}{2}).
\]
\end{proof}

\end{document}